\documentclass[11pt,A4paper]{amsart}

\usepackage[portrait, margin=2.2cm, footskip=1cm]{geometry}
\date{}
\usepackage[foot]{amsaddr}
\usepackage{graphicx}
\usepackage{amstext, amssymb}
\usepackage{epic, eepic}
\usepackage{enumitem}
\usepackage[center]{caption2}
\usepackage{subfigure}
\usepackage{setspace}
\allowdisplaybreaks[4]
\DeclareMathOperator{\tr}{tr}
\newtheorem{theorem}{Theorem}
\newtheorem{proposition}{Proposition}

\newtheorem{result}{Result}
\newtheorem{lemma}{Lemma}
\theoremstyle{definition}
\newtheorem{definition}{Definition}
\newtheorem{remark}{Remark}
\makeatletter

\begin{document}
\setlength{\baselineskip}{18pt}

\title{Relationships among quasivarieties induced by the min networks on inverse semigroups}
\author{Ying-Ying Feng $^{1,3}$}
\email[Ying-Ying Feng]{rickyfungyy@fosu.edu.cn, yingying.feng@york.ac.uk}
\author{Li-Min Wang $^2$}
\author{Zhi-Yong Zhou $^{2,4}$}
\address{$^1$ Department of Mathematics, Foshan University, Foshan 528000, P. R. China}
\address{$^2$ School of Mathematical Sciences, South China Normal University, Guangzhou 510631, P. R. China}
\address{$^3$ Department of Mathematics, University of York, York YO10 5DD, UK}
\address{$^4$ Guangzhou iFLY ZUNHONG Information Tech Co., Ltd., Guangzhou 510665, P. R. China}
\keywords{Inverse semigroup, min network of congruences, $\ker{\alpha_n}$-is-Clifford semigroup, $\beta_n$-is-over-$E$-unitary semigroup, relationship}
\maketitle

\begin{abstract}
 A congruence on an inverse semigroup $S$ is determined uniquely by its kernel and trace. Denoting by $\rho_k$ and $\rho_t$ the least congruence on $S$ having the same kernel and the same trace as $\rho$, respectively, and denoting by $\omega$ the universal congruence on $S$, we consider the sequence $\omega$, $\omega_k$, $\omega_t$, $(\omega_k)_t$, $(\omega_t)_k$, $((\omega_k)_t)_k$, $((\omega_t)_k)_t$, $\cdots$. The quotients $\{S/\omega_k\}$, $\{S/\omega_t\}$, $\{S/(\omega_k)_t\}$, $\{S/(\omega_t)_k\}$, $\{S/((\omega_k)_t)_k\}$, $\{S/((\omega_t)_k)_t\}$, $\cdots$, as $S$ runs over all inverse semigroups, form quasivarieties. This article explores the relationships among these quasivarieties.
\end{abstract}

\section*{Introduction}

The study of congruences has always been one of the cornerstones of semigroup theory, and a major strand in this direction has been the description of the congruence lattices of specific semigroups or families of semigroups. Many classical results appear in monographs, see for example \cite{inverse, completely}. Recent work includes a study of congruence lattices of diagram monoids \cite{em2018, er2018}, nilsemigroups \cite{pj2017}, ample semigroups \cite{q2019} and graph inverse semigroups \cite{w2019}. Meanwhile, algebraists are digging deeper into the lattice of congruences of inverse semigroups, see, for example,   \cite{b2019, fw2019, wf2011}. One of the established treatments of congruences on inverse semigroups is the kernel--trace approach, which consists in splitting the analysis of a congruence $\rho$ on an inverse semigroup $S$ into two parts:  the \emph{kernel}, that is, the set of elements, $\rho$-related to idempotents, and the \emph{trace}, that is, the restriction of $\rho$ to the set of idempotents of $S$. From the kernel--trace decomposition of congruences, we obtain two operators, lower $k$ and lower $t$, on the congruence lattice of an inverse semigroup. We denote by $\rho_k$ the least congruence on $S$ having the same kernel as $\rho$, and by $\rho_t$ the least congruence having the same trace as $\rho$. On a fixed inverse semigroup $S$, starting with the universal congruence $\omega$, we form two sequences: 
\[ \omega,~\omega_k,~(\omega_k)_t,~((\omega_k)_t)_k,~\cdots \quad \text{and} \quad \omega,~\omega_t,~(\omega_t)_k,~((\omega_t)_k)_t,~\cdots,\] written more simply as 

\begin{equation}\label{eqn:sequence}\omega,~\omega_k,~\omega_{kt}, ~\omega_{ktk},~\cdots \quad \text{and} \quad \omega,~\omega_t,~\omega_{tk}, ~\omega_{tkt},~\cdots,\end{equation} and often denoted by 
\begin{equation}\label{eqn:sequence2}\alpha_0,~\beta_1,~\alpha_2,~\beta_3,~\cdots \quad \text{and} \quad \beta_0,~\alpha_1,~\beta_2,~\alpha_3,~\cdots;\end{equation}
having both labelling conventions transpires to be useful.
These congruences, known as the {\em min network} of congruences on $S$, together with the intersections $\alpha_1 \cap \beta_1$, $\alpha_2 \cap \beta_2$, $\alpha_3 \cap \beta_3$, $\cdots$, form a sublattice of the lattice of all congruences on $S$ \cite{network}. Petrich and Reilly \cite{network} first investigated properties of these congruences. For all but the initial elements of the sequences, the resulting quotient semigroups lie in particular quasivarieties and \cite{network} establishes a system of defining implications. Indeed, $\alpha_1=\sigma$, $\beta_1=\eta$, $\alpha_2=\nu$, $\beta_2=\pi$, $\beta_3=\lambda$ are, respectively, the least congruences $\rho$ on $S$ such that $S/\rho$ is a group, semilattice, Clifford semigroup, $E$-unitary inverse semigroup, and $E$-reflexive inverse semigroup.

Notice that there are repeated patterns in the quasivarieties to which $S/\alpha_2$, $S/\alpha_1$, $S/\beta_3$ and $S/\beta_2$ belong. A Clifford semigroup ($S/\alpha_2)$ is a semilattice of groups ($S/\alpha_1)$, and an $E$-reflexive inverse semigroup ($S/\beta_3)$ is a semilattice of $E$-unitary inverse semigroups ($S/\beta_2$). Another pattern  appears  in the quasivarieties to which $S/\beta_2$, $S/\beta_1$, $S/\alpha_3$ and $S/\alpha_2$ belong \cite{wf2011}. The closure of the set of idempotents, or $E\omega$,
of 
an $E$-unitary inverse semigroup ($S/\beta_2$)
is a semilattice ($S/\beta_1$), and $E\omega$ of $S/\alpha_3$ is  Clifford ($S/\alpha_2)$. Naturally,  we wonder whether these patterns continue indefinitely.

In the process of achieving this goal, we did not figure out the relationship from start to end of the sequence (\ref{eqn:sequence}) or one by one. Instead, a different technique is adopted. Notice that the classes of quotient semigroups by congruences in the min network are essentially characterised by congruence equations such as  $\omega_t=\varepsilon$, $\omega_{tk}=\varepsilon$, $\omega_k=\varepsilon$, $\omega_{kt}=\varepsilon$, and $\omega_{ktk}=\varepsilon$. For example, $S$ is a Clifford  [resp. $E$-reflexive inverse]  semigroup if and only if $\omega_{kt}=\varepsilon$ [resp. $\omega_{ktk}=\varepsilon$]. In this article we make a general investigation of the properties of congruences $\rho$ satisfying $\rho_{kt}=\varepsilon$, $\rho_{tk}=\varepsilon$, or $\rho_{tkt}=\varepsilon$, which helps us not only determine the quasivarieties to which the quotient semigroups belong, but also obtain relationships among the corresponding quasivarieties.

In Section 1 we summarise the notation and terminology necessary used in this paper. In Section 2 we provide expressions for minimal congruences on quotient semigroups by the min network of congruences on a regular semigroup. Our principal results are in Section 3, where we explicate the relationships among quasivarieties related to the min network on an inverse semigroup. In the process of proving these results, we investigate properties of congruences satisfying $\rho_{kt}=\varepsilon$, $\rho_{tk}=\varepsilon$, or $\rho_{tkt}=\varepsilon$.

\section{Preliminaries}

\emph{Throughout the entire paper, $S$ stands for a regular semigroup.}

We shall use the notation and terminology of Howie \cite{howie} and Petrich \cite{inverse}, to which the reader is referred for basic
information and results on inverse semigroups. For an arbitrary regular semigroup $S$, we denote by $E_S$ its set of idempotents, and by $V(x)$ the set of all inverses of an element $x\in S$. The unique inverse of an element $x$ of an inverse semigroup is denoted by $x^{-1}$. For every relation $\mathsf{R}$ on $S$ we denote by $\mathsf{R}^\infty$ the smallest transitive relation on $S$ containing $\mathsf{R}$, and by $\mathsf{R}^*$ the least congruence on $S$ containing $\mathsf{R}$. If $Z$ is a subsemigroup of  $S$, we denote by $\mathsf{R}|_Z$ the restriction of $\mathsf{R}$ to $Z$, and if $\mathsf{R}$ is a specific relation (such as, a Green's relation or the least group congruence), then we use the notation $\mathsf{R}_Z$ for the corresponding relation defined on  $Z$.

The centralizer of $E_S$ in $S$, $E_S\zeta$,  is defined by $$E_S\zeta=\{a \in S\,|\,ae=ea~\text{for all}~e \in E_S\}.$$ 
The closure of $E_S$ in $S$, $E_S\omega$,  is defined by $$E_S\omega=\{a \in S\,|\,a \ge e~\text{for some}~e \in E_S\}.$$ Here $\ge$ denotes the natural partial order on $S$, which is defined by the rule that $$a \le b \Leftrightarrow (\exists\,e,f \in E_S)~a=eb 
\mbox{ and }a=bf;$$
if $S$ is  inverse, then we need only one of these two conditions.  
The regular semigroup
$S$ is {\em Clifford}  if its idempotents lie in its centre; equivalently, by  a well-known structure theorem, $S$  is a (strong) semilattice of groups.  The semigroup $S$ is said to be \emph{$E$-unitary} if $ey=e$ for some $e \in E_S$ implies that $y \in E_S$. Equivalently, $S$ is $E$-unitary if and only if it satisfies the implication $xy=x \Rightarrow y^2=y$.

For any regular semigroup $S$, let $\mathcal{C}(S)$ be the lattice of all congruences on $S$. For $\rho \in \mathcal{C}(S)$, $\ker{\rho}=\{a \in S\,|\,a\, \rho\, e~ \text{for some}~e \in E_S\}$ is the \emph{kernel} of $\rho$, and $\tr{\rho}=\rho|_{_{E_S}}$ is the \emph{trace} of $\rho$. A congruence on a regular semigroup or an inverse semigroup is determined uniquely by its kernel and trace.

\begin{result}\textup{(\cite[Lemma 2.10]{pp1986})} Let $\rho$ be a congruence on a regular  semigroup $S$. Then $$a~\rho~b \iff a\,[\mathcal{L}(\tr{\rho})\mathcal{L}(\tr{\rho})\mathcal{L} \cap \mathcal{R}(\tr{\rho})\mathcal{R}(\tr{\rho})\mathcal{R}]\,b, ~ ab' \in \ker{\rho}~\text{for some [all] } b' \in V(b).$$
\end{result}

\begin{result}\textup{(\cite[Theorem 4.4]{congruence})}
 Let $\rho$ be a congruence on an inverse semigroup $S$. Then $$a~\rho~b \iff a^{-1}a\,\tr{\rho}\, b^{-1}b,~ab^{-1} \in \ker{\rho}.$$
\end{result}
Consequently, for $\rho, \theta \in \mathcal{C}(S)$, $$\rho \subseteq \theta \iff \tr{\rho} \subseteq \tr{\theta},~\ker{\rho} \subseteq \ker{\theta}.$$

For any $\rho \in \mathcal{C}(S)$, the relations $\mathcal{T}$ and $\mathcal{K}$ are defined as follows, $$\rho~\mathcal{T}~\theta
\iff \tr{\rho}=\tr{\theta}, \qquad \rho~\mathcal{K}~\theta \iff \ker{\rho}=\ker{\theta}.$$ The relation $\mathcal{T}$ is a complete congruence on the lattice $\mathcal{C}(S)$, while $\mathcal{K}$ is an equivalence relation on $\mathcal{C}(S)$ \cite{inverse}. The equivalence class $\rho \mathcal{T}$ [resp. $\rho \mathcal{K}$] is an interval of $\mathcal{C}(S)$ with greatest and least element to be denoted by $\rho^T$ [resp. $\rho^K$] and $\rho_t$ [resp. $\rho_k$], respectively.

\begin{result}\textup{(\cite[Theorem 3.2]{pp1986})}
 For any congruence $\rho$ on a regular semigroup $S$, $$\rho_t=(\tr{\rho})^*, \quad \rho_k=\{(x,x^2)|\,x \in \ker{\rho}\}^*.$$
\end{result}

There is an explicit expression for $\rho_t$ on inverse semigroups.

\begin{result}\textup{(\cite[Theorem \@Roman3.2.5]{inverse})}\label{rholt}
 For any congruence $\rho$ on an inverse semigroup $S$, $$a~\rho_t~b \iff (\exists\,e \in E_S)~ae=be \text{ and } e\,\rho\,a^{-1}a\,\rho\,b^{-1}b.$$
\end{result}

On any inverse semigroup $S$, the relation $\mathcal{F}$ is defined by $$a\, \mathcal{F}\, b \iff a^{-1}b \in E_S.$$

\begin{remark}\label{f}
 The relation $\mathcal{F}$ is reflexive, symmetric and compatible. First, $\mathcal{F}$ is reflexive, since $a^{-1}a \in E_S$ for every $a \in S$ and hence $a\,\mathcal{F}\,a$. Second, $\mathcal{F}$ is symmetric, since $b^{-1}a=(a^{-1}b)^{-1}\in E_S$ if $a\,\mathcal{F}\,b$ and thus $b\,\mathcal{F}\,a$. Finally, $\mathcal{F}$ is compatible, since 
 if $a\,\mathcal{F}\,b$  we have $(ac)^{-1}(bc)=c^{-1}(a^{-1}b)c \in E_S$ and $(ca)^{-1}(cb)=a^{-1}c^{-1}cb=(a^{-1}b)(b^{-1}c^{-1}cb) \in E_S$  and so $ac\,\mathcal{F}\,bc$ and $ca\,\mathcal{F}\,cb$.
\end{remark}

There are alternative expressions for the extremal congruences on inverse semigroups.

\begin{result}\label{tkgen} \textup{(\cite[Theorem 6.2]{network})}
 For any congruence $\rho$ on an inverse semigroup $S$, $$\rho_t=(\rho \cap \mathcal{F})^*, \quad \rho_k=(\rho \cap \mathcal{L})^*=(\rho \cap \mathcal{R})^*.$$
\end{result}

Let $\mathcal{P}$ be a class of semigroups and $\rho \in \mathcal{C}(S)$. Then $\rho$ is \emph{over $\mathcal{P}$} if each $\rho$-class which
is a subsemigroup of $S$ belongs to $\mathcal{P}$ and  $\rho$ is a \emph{$\mathcal{P}$-congruence} if $S/\rho \in \mathcal{P}$. For example, $\mathcal{J}$ is over groups on Clifford semigroups. A congruence $\rho$ on $S$ is \emph{idempotent separating} if $e,f\in E_S$  and $e\, \rho\, f$ imply that $e=f$. On the other hand, $\rho$ is \emph{idempotent pure} if $\ker{\rho}=E_S$. We denote by $\mu$ and $\tau$ the greatest idempotent separating and greatest idempotent pure congruences on $S$, respectively. The equality and the universal relations on $S$ are denoted by $\varepsilon$ and $\omega$ respectively.

The following result will frequently be useful.

\begin{result}\textup{(\cite[Theorem 1.5.4]{howie})}
 Let $\rho$, $\theta$ be congruences on an arbitrary semigroup $S$ such that $\rho \subseteq \theta$. Then $\theta/\rho=\{(x\rho, y\rho) \in S/\rho \times S/\rho\,|\,(x,y) \in \theta\}$ is a congruence on $S/\rho$, and $(S/\rho)/(\theta/\rho) \simeq S/\theta$. In particular, $\varepsilon_{S/\rho}=\rho/\rho$ and $\omega_{S/\rho}=\omega/\rho$.
\end{result}

Properties of congruences which are obtained by a recursive process based on the concepts of the kernel and the trace of a congruence were first studied by Petrich -- Reilly \cite{network}.

\begin{definition}\textup{(\cite[Definition 5.1]{network})}\label{not}
 On a regular semigroup $S$ we define inductively the following two sequences of congruences:
 \begin{eqnarray*}
  &&\alpha_0=\omega=\beta_0,\\
  &&\alpha_n=(\beta_{n-1})_t, \quad \beta_n=(\alpha_{n-1})_k \quad \text{for~} n \geqslant 1.
 \end{eqnarray*}
 We call the aggregate $\{\alpha_n, \beta_n\}_{n=0}^\infty$, together with the inclusion relation for congruences, the \emph{min network} of
 congruences on $S$.
\end{definition}

The min network of congruences on inverse semigroups is related to the following classes of semigroups.
\begin{definition}
 An inverse semigroup $S$ for which $\ker{\alpha_n}$ is a Clifford semigroup is called a \emph{$\ker{\alpha_n}$-is-Clifford} \emph{semigroup}. An inverse semigroup $S$ is called a \emph{$\beta_n$-is-over-$E$-unitary semigroup} if $e\beta_n$ is $E$-unitary for each $e \in E_S$.
\end{definition}

\begin{remark}
The class of $\ker{\alpha_n}$-is-Clifford semigroups forms a quasivariety. So does the class of $\beta_n$-is-over-$E$-unitary semigroups \cite{network}.
\end{remark}

The next result develops some basic facts about the min network.
\begin{result}\textup{(\cite[Theorem 2.9, Theorem 2.13]{fw2019})}
 Let $S$ be an inverse semigroup and let $n$ be a non-negative integer. Then
 \begin{enumerate}[label=(\arabic*)]
  \item $\alpha_{n+2}$ is the minimum congruence $\rho$ on $S$ such that $S/\rho$ is a $\ker{\alpha_n}$-is-Clifford semigroup;
  \item $\beta_{n+2}$ is the minimum congruence $\rho$ on $S$ such that $S/\rho$ is a $\beta_n$-is-over-$E$-unitary semigroup.
 \end{enumerate}
\end{result}

The min network is depicted in Figure \ref{fig} together with the types of semigroups to which the quotient semigroups belong.
\begin{figure}[!hbt]\label{fig}
 \renewcommand*\figurename{Figure}
 \renewcommand*\captionlabeldelim{}
 \setlength{\unitlength}{0.7cm}
 \begin{center}
  \begin{picture}(4,24)
   \drawline(1,3)(4,6)(0,10)(4,14)(0,18)(4,22)(2,24)(0,22)(4,18)(0,14)(4,10)(0,6)(4,2)(3,1)
   \allinethickness{0.7mm}
   \put(0,6){\circle*{0.06}} \put(0,10){\circle*{0.06}} \put(0,14){\circle*{0.06}} \put(0,18){\circle*{0.06}} \put(0,22){\circle*{0.06}}
   \put(2,24){\circle*{0.06}} \put(2,20){\circle*{0.06}} \put(2,16){\circle*{0.06}} \put(2,12){\circle*{0.06}} \put(2,8){\circle*{0.06}}
   \put(2,4){\circle*{0.06}} \put(4,2){\circle*{0.06}} \put(4,6){\circle*{0.06}} \put(4,10){\circle*{0.06}} \put(4,14){\circle*{0.06}}
   \put(4,18){\circle*{0.06}} \put(4,22){\circle*{0.06}}
   \put(2,24.3){\makebox(0,0)[b]{$\omega$}}
   \put(-0.3,22){\makebox(0,0)[r]{$\sigma=\alpha_1$}}
   \put(-0.3,18){\makebox(0,0)[r]{$\eta_t=\nu=\alpha_2$}}
   \put(-0.3,14){\makebox(0,0)[r]{$\pi_t=\alpha_3$}}
   \put(-0.3,10){\makebox(0,0)[r]{$\lambda_t=\alpha_4$}}
   \put(-0.3,6){\makebox(0,0)[r]{$(\beta_4)_t=\alpha_5$}}
   \put(4.3,22){\makebox(0,0)[l]{$\beta_1=\eta$}}
   \put(4.3,18){\makebox(0,0)[l]{$\beta_2=\pi=\sigma_k$}}
   \put(4.3,14){\makebox(0,0)[l]{$\beta_3=\lambda=\nu_k$}}
   \put(4.3,10){\makebox(0,0)[l]{$\beta_4=\pi_{tk}$}}
   \put(4.3,6){\makebox(0,0)[l]{$\beta_5=\lambda_{tk}$}}
   \put(4.3,2){\makebox(0,0)[l]{$\beta_6=(\alpha_5)_k$}}
   \put(-0.3,21.2){\makebox(0,0)[r]{\footnotesize{\textsf{group}}}}
   \put(-0.3,17.2){\makebox(0,0)[r]{\footnotesize{\textsf{Clifford}}}}
   \put(-0.3,13.2){\makebox(0,0)[r]{\footnotesize{$\ker{\sigma}$-\textsf{is-Clifford}}}}
   \put(-0.3,9.2){\makebox(0,0)[r]{\footnotesize{$\ker{\nu}$-\textsf{is-Clifford}}}}
   \put(-0.3,5.2){\makebox(0,0)[r]{\footnotesize{$\ker{\alpha_3}$-\textsf{is-Clifford}}}}
   \put(4.3,21.2){\makebox(0,0)[l]{\footnotesize{\textsf{semilattice}}}}
   \put(4.3,17.2){\makebox(0,0)[l]{\footnotesize{$E$-\textsf{unitary}}}}
   \put(4.3,13.2){\makebox(0,0)[l]{\footnotesize{$E$-\textsf{reflexive}}}}
   \put(4.3,9.2){\makebox(0,0)[l]{\footnotesize{$\pi$-\textsf{is-over-$E$-unitary}}}}
   \put(4.3,5.2){\makebox(0,0)[l]{\footnotesize{$\lambda$-\textsf{is-over-$E$-unitary}}}}
   \put(4.3,1.2){\makebox(0,0)[l]{\footnotesize{$\beta_4$-\textsf{is-over-$E$-unitary}}}}
   \put(2,0.5){\makebox(0,0)[c]{$\vdots$}}
  \end{picture}
  \caption{\quad Min network of congruences on inverse semigroups} \label{origin1}
 \end{center}
\end{figure}
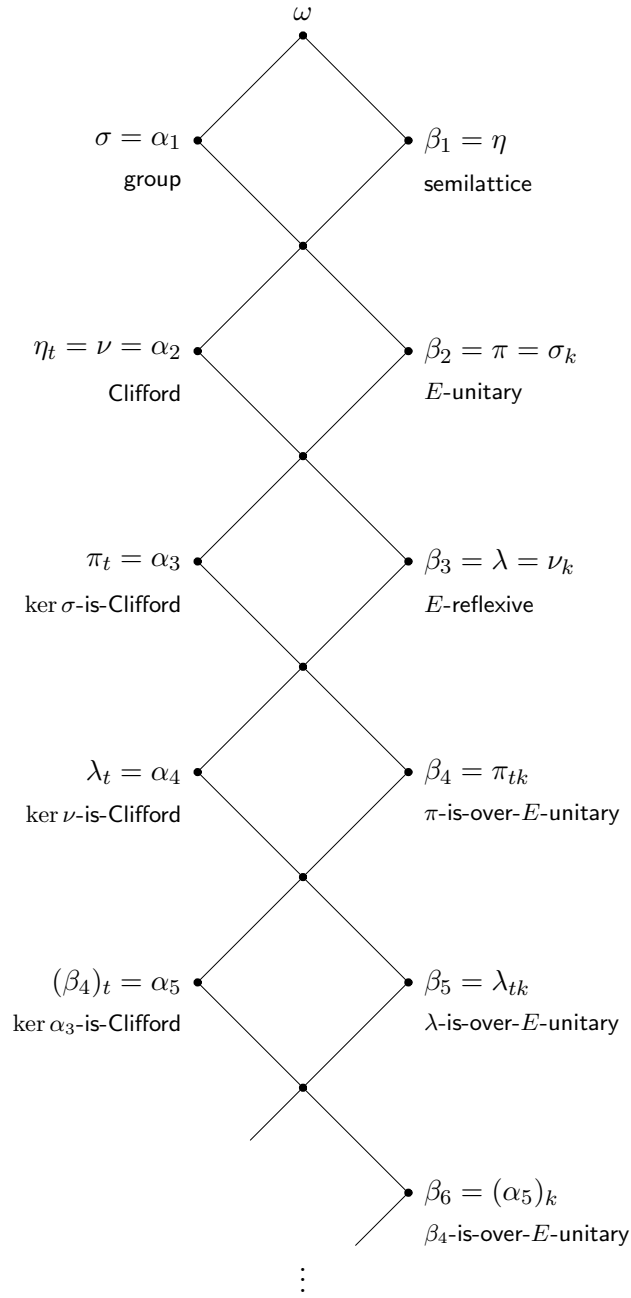

\section{Minimal congruences on quotient semigroups}

After formulating the least congruences having the same trace or kernel as a given congruence on the quotient of a regular semigroup, we provide expressions for the min network of congruences on quotient semigroups corresponding to the min network of congruences on a regular semigroup. We then characterise congruences $\rho$ with the property $\rho_t=\varepsilon$ and $\rho_k=\varepsilon$.

We shall need some auxiliary results first.

\begin{lemma}\label{least}
 Let $\rho$, $\theta$ be congruences on a regular semigroup $S$. Then
 \begin{enumerate}[label=(\arabic*)]
  \item $\tr{\rho} \subseteq \tr{\theta} \iff \rho_t \subseteq \theta$;
  \item $\ker{\rho} \subseteq \ker{\theta} \iff \rho_k \subseteq \theta$.
 \end{enumerate}
\end{lemma}
\begin{proof}
 (1) If $\tr{\rho} \subseteq \tr{\theta}$, then $\tr{\rho} \subseteq \theta$ and $\rho_t=(\tr{\rho})^* \subseteq \theta$. Conversely, suppose that $\rho_t \subseteq \theta$. Then $\tr{\rho}=\tr{\rho_t} \subseteq \tr{\theta}$.

 (2) If $\ker{\rho} \subseteq \ker{\theta}$, then $$\{(x,x^2)\,|\,x \in \ker{\rho}\} \subseteq \{(x,x^2)\,|\,x \in \ker{\theta}\} \subseteq \theta,$$ and so $\rho_k=\{(x,x^2)\,|\,x \in \ker{\rho}\}^* \subseteq \theta$. Conversely, if $\rho_k \subseteq \theta$, then $\ker{\rho}=\ker{\rho_k} \subseteq \ker{\theta}$.
\end{proof}

\begin{lemma}\label{quotient}
 Let $\rho$, $\gamma$, $\theta$ be congruences on a regular semigroup $S$ such that $\rho \subseteq \gamma$ and $\rho \subseteq \theta$. Then
 \begin{enumerate}[label=(\arabic*)]
  \item $\tr{\gamma/\rho}=\tr{\theta/\rho} \iff \tr{\gamma}=\tr{\theta}$;
  \item $\ker{\gamma/\rho}=\ker{\theta/\rho} \iff \ker{\gamma}=\ker{\theta}$.
 \end{enumerate}
 In particular, if $\rho \subseteq \theta$ and $\tr{\rho}=\tr{\theta}$, then $\tr{\theta/\rho}=\tr{\varepsilon_{S/\rho}}$; if $\rho \subseteq \theta$ and $\ker{\rho}=\ker{\theta}$, then $\ker{\theta/\rho}=\ker{\varepsilon_{S/\rho}}$.
\end{lemma}
\begin{proof}
 (1) Assume that $\tr{\gamma/\rho}=\tr{\theta/\rho}$, and let $e$, $f \in E_S$ be $\gamma$-related. Then $e\rho\,(\gamma/\rho)\,f\rho$ and thus $e\rho\,(\theta/\rho)\,f\rho$, since $\tr{\gamma/\rho}=\tr{\theta/\rho}$. Therefore, $e\,\theta\,f$ whence $\tr{\gamma} \subseteq \tr{\theta}$. In a similar way we may find that $\tr{\theta} \subseteq \tr{\gamma}$ and hence $\tr{\gamma}=\tr{\theta}$.

 Conversely, we suppose that $\tr{\gamma}=\tr{\theta}$ and let $e$, $f \in E_S$ be such that $e\rho\,(\gamma/\rho)\,f\rho$. Then $e\,\gamma\,f$ and so by $\tr{\gamma}=\tr{\theta}$ we get first that $e\,\theta\,f$ and then that $e\rho\,(\theta/\rho)\,f\rho$, from which it follows that $\tr{\gamma/\rho} \subseteq \tr{\theta/\rho}$. In the same way we can show that $\tr{\theta/\rho} \subseteq \tr{\gamma/\rho}$ and thus $\tr{\gamma/\rho}=\tr{\theta/\rho}$.

 (2) Assume that $\ker{\gamma/\rho}=\ker{\theta/\rho}$. If $a \in \ker{\gamma}$, then there exists $e \in E_S$ such that $a\,\gamma\,e$. Thus $a\rho\,(\gamma/\rho)\,e\rho$ and $a\rho \in \ker{\gamma/\rho}=\ker{\theta/\rho}$, which implies that there exists $f \in E_S$ such that $a\rho\,(\theta/\rho)\,f\rho$. Hence $a\,\theta\,f$ and $a \in \ker{\theta}$, whence $\ker{\gamma} \subseteq \ker{\theta}$. By symmetry we may conclude that $\ker{\theta} \subseteq \ker{\gamma}$ and so $\ker{\gamma}=\ker{\theta}$.

 Conversely, suppose that $\ker{\gamma}=\ker{\theta}$. If $a\rho \in \ker{\gamma/\rho}$, then there exists $e \in E_S$ such that $a\rho\,(\gamma/\rho)\,e\rho$. Thus $a\,\gamma\,e$ whence it also follows that $a \in \ker{\gamma}=\ker{\theta}$. Now there exists $f \in E_S$ such that $a\,\theta\,f$. We have $a\rho\,(\theta/\rho)\,f\rho$ so that $a\rho \in \ker{\theta/\rho}$. Therefore $\ker{\gamma/\rho} \subseteq \ker{\theta/\rho}$. A similar argument establishes that $\ker{\theta/\rho} \subseteq \ker{\gamma/\rho}$ and so $\ker{\gamma/\rho}=\ker{\theta/\rho}$.

 In particular, if $\rho \subseteq \theta$ and $\tr{\rho}=\tr{\theta}$, then $\tr{\theta/\rho}=\tr{\rho/\rho}=\tr{\varepsilon_{S/\rho}}$; if $\rho \subseteq \theta$ and $\ker{\rho}=\ker{\theta}$, then $\ker{\theta/\rho}=\ker{\rho/\rho}=\ker{\varepsilon_{S/\rho}}$.
\end{proof}

The next result gives some information about the behaviour of the lower end of the trace and kernel classes concerning quotients.

\begin{proposition}\label{min}
 Let $\rho$, $\theta$ be congruences on a regular semigroup $S$ such that $\rho \subseteq \theta$. Then $$(\theta/\rho)_t=(\rho \vee \theta_t)/\rho, \qquad (\theta/\rho)_k=(\rho \vee \theta_k)/\rho.$$ In particular, if $\rho \subseteq \theta_t$, then $(\theta/\rho)_t=\theta_t/\rho$; if $\rho \subseteq \theta_k$, then $(\theta/\rho)_k=\theta_k/\rho$.
\end{proposition}
\begin{proof}
 Assume that $\rho \subseteq \theta$. We first establish that $\tr{(\rho \vee \theta_t)/\rho}=\tr{\theta/\rho}$. To do this it will be sufficient to show that $\tr{(\rho \vee \theta_t)}=\tr{\theta}$ in view of Lemma \ref{quotient}. Notice that $\theta_t \subseteq \rho \vee \theta_t \subseteq \theta \vee \theta_t=\theta$, whence $\tr{\theta}=\tr{\theta_t} \subseteq \tr{(\rho \vee \theta_t)} \subseteq \tr{\theta}$. It follows that $\tr{(\rho \vee \theta_t)}=\tr{\theta}$ and $\tr{(\rho \vee \theta_t)/\rho}=\tr{\theta/\rho}$, whence $(\theta/\rho)_t \subseteq (\rho \vee \theta_t)/\rho$.

 Next, consider $(\rho \vee \theta_t)/\rho \subseteq (\theta/\rho)_t$. Let $\gamma \in \mathcal{C}(S)$ be such that $\rho \subseteq \gamma$ and $\gamma/\rho=(\theta/\rho)_t$. Then $\tr{\gamma/\rho}=\tr{\theta/\rho}$. It follows from Lemma \ref{quotient} that $\tr{\gamma}=\tr{\theta}$. Thus $\theta_t \subseteq \gamma$ and so $\rho \vee \theta_t \subseteq \gamma$, whence $(\rho \vee \theta_t)/\rho \subseteq \gamma/\rho=(\theta/\rho)_t$, and the required equality holds.

 An exactly parallel argument establishes the assertion concerning $(\theta/\rho)_k$ and so the proof is omitted.
\end{proof}

We are now ready for a description of the min network on the quotient semigroup.

\begin{proposition}\label{ab}
 Let $S$ be a regular semigroup, $n$ be a positive integer and $i$ be a non-negative integer. If $i \leqslant n$, then
 \begin{enumerate}[label=(\arabic*)]
 \item $(\alpha_i)_{S/\alpha_n}=\alpha_i/\alpha_n$, $(\beta_i)_{S/\alpha_n}=\beta_i/\alpha_n$;
 \item $(\alpha_i)_{S/\beta_n}=\alpha_i/\beta_n$, $(\beta_i)_{S/\beta_n}=\beta_i/\beta_n$.
 \end{enumerate}
\end{proposition}
\begin{proof}
 (1) We will first observe that the statement is true for $i=0$, and then complete the proof with an inductive argument.

 If $i=0$, then $\alpha_0=\omega$ and $\alpha_0/\alpha_n=\omega/\alpha_n=\omega_{S/\alpha_n}=(\alpha_0)_{S/\alpha_n}$. Similarly, we have $(\beta_0)_{S/\alpha_n}=\beta_0/\alpha_n$.

 Now suppose that the statement is valid for $i<n$. Then $\alpha_n \subseteq \alpha_{i+1}=(\beta_i)_t$, and hence
 \begin{align*}
  \alpha_{i+1}/\alpha_n=(\beta_i)_t/\alpha_n&=(\beta_i/\alpha_n)_t &&\text{by Proposition \ref{min}}\\
  &=((\beta_i)_{S/\alpha_n})_t &&\text{by induction hypothesis}\\
  &=(\alpha_{i+1})_{S/\alpha_n}.&&
 \end{align*}
 Similarly we may have $(\beta_{i+1})_{S/\alpha_n}=\beta_{i+1}/\alpha_n$.

 (2) The proof is closely similar to that for (1) and is omitted.
\end{proof}

The next two lemmas provide characterisations for congruences having the property that $\rho_t=\varepsilon$ or $\rho_k=\varepsilon$. We will later perform an analogous analysis to congruences having similar properties.

\begin{lemma}\label{t}
 The following statements concerning a congruence $\rho$ on a regular semigroup $S$ are equivalent.
 \begin{enumerate}[label=(\arabic*)]
  \item $\rho_t=\varepsilon$;
  \item $\rho \subseteq \mu$;
  \item $\tr{\rho}=\varepsilon|_{E_S}$;
  \item $\rho$ is idempotent separating;
  \item for every $e \in E_S$, $e\rho$ is a group;
  \item $\rho \subseteq \mathcal{H}$.
 \end{enumerate}
\end{lemma}
\begin{proof}
 The equivalence of (3), (4), (5) and (6) follows from \cite[Proposition 2.4.5]{howie} and \cite[Corollary 2.21]{pp1986}.

 $(1) \Rightarrow (3)$. Let $\rho_t=\varepsilon$. Then $\tr{\rho}=\tr{\rho_t}=\rho_t|_{E_S}=\varepsilon|_{E_S}$.

 $(3) \Rightarrow (2)$. By $\tr{\rho}=\varepsilon|_{E_S}$ we have that $\rho$ is idempotent separating. Hence $\rho \subseteq \mu$ since $\mu$ is the greatest idempotent separating congruence.

 $(2) \Rightarrow (1)$. From $\rho \subseteq \mu$ we infer that $\rho_t \subseteq \mu_t=\varepsilon$ which implies that $\rho_t=\varepsilon$.
\end{proof}

\begin{lemma}\label{kreg}
 The following statements concerning a congruence $\rho$ on a regular semigroup $S$ are equivalent.
 \begin{enumerate}[label=(\arabic*)]
  \item $\rho_k=\varepsilon$;
  \item $\rho \subseteq \tau$;
  \item $\ker{\rho}=E_S$;
  \item $\rho$ is idempotent pure;
  \item for every $e \in E_S$, $e\rho$ is a band;
  \item $\rho$ satisfies the implication $x^2\,\rho\,x \Rightarrow x^2=x$.
 \end{enumerate}
\end{lemma}
\begin{proof}
 The equivalence of (1), (3), (4) and (5) follows from \cite[Proposition 3.1]{pp1988}.

 $(2) \Rightarrow (3)$. From $\rho \subseteq \tau$ we infer that $\ker{\rho} \subseteq \ker{\tau}=E_S$ which implies that $\ker{\rho}=E_S$.

 $(3) \Rightarrow (6)$. Let $x \in S$ be such that $x^2\,\rho\,x$. Then $x \in \ker{\rho}=E_S$ which yields $x^2=x$.

 $(6) \Rightarrow (2)$. Suppose that $x \in \ker{\rho}$. Then $x^2\,\rho\,x$ and thus $x^2=x$, which implies that $\ker{\rho} \subseteq E_S$ and thus $\ker{\rho}=E_S$. Therefore $\rho \subseteq \tau$ since $\tau$ is the greatest idempotent pure congruence.
\end{proof}

\section{Relationships among the quasivarieties}

We will now explore relationships among the quasivarieties related to the min network on inverse semigroups. After providing a description of the idempotent $\rho_t$-classes, we present characterisations for congruences satisfying $\rho_{kt}=\varepsilon$, $\rho_{tk}=\varepsilon$, or $\rho_{tkt}=\varepsilon$. We then reach a solution to the problem proposed in \cite{wf2011}. \emph{Throughout this section, unless otherwise stated, $S$ denotes an arbitrary inverse semigroup with semilattice $E_S$ of idempotents.}

For a given congruence $\rho$ on $S$, it is known that $\rho_t$ is uniquely determined by its kernel $\ker{\rho_t}$ and its trace $\tr{\rho_t}$. But $\tr{\rho_t}=\tr{\rho}$ and hence in order to characterise $\rho_t$, it suffices to provide a description of $\ker{\rho_t}$. Noticing that $\ker{\rho_t}$ is a union of the idempotent $\rho_t$-classes, we establish a characterisation for the idempotent $\rho_t$-classes in virtue of a description for the least group congruence on an inverse semigroup, $$\sigma=\{(x,y) \in S \times S\,|\,(\exists\,e \in E_S)~xe=ye\}.$$ The following proposition is a key result in the article.

\begin{proposition}\label{kernel}
 Let $\rho$ be a congruence on an inverse semigroup $S$. Then for any $e \in E_S$ we have $$\ker{\sigma_{e\rho}}=e\rho_t.$$
\end{proposition}
\begin{proof}
 If $x \in \ker{\sigma_{e\rho}}$, then $x^{-1} \in \ker{\sigma_{e\rho}}$ and $x^{-1}\,\sigma_{e\rho}\,f$ for some $f \in E_S$, so $x^{-1}g=fg$ for some $g \in E_{e\rho}$, and thus $x^{-1}g \in E_S$ while $x\,\rho\,e\,\rho\,g$, that is, $x\,(\rho \cap \mathcal{F})\,g$. By Lemma \ref{tkgen} we have $x\,\rho_t\,g$. But $g\,\rho\,e$ and so $g\,\rho_t\,e$, since $\tr{\rho}=\tr{\rho_t}$. It follows that $x\,\rho_t\,e$. Thus $x \in e\rho_t$ and so $\ker{\sigma_{e\rho}} \subseteq e\rho_t$.

 Conversely, assume that $x \in e\rho_t$. In view of Remark \ref{f} we know that $\rho \cap \mathcal{F}$ is reflexive, symmetric and compatible and so $\rho_t=(\rho \cap \mathcal{F})^*=(\rho \cap \mathcal{F})^\infty$ by \cite[Proposition 1.5.8, Lemma 1.5.5, Proposition 1.4.9]{howie}. Then there exist $z_1$, $z_2$, $\cdots$, $z_n \in S$ such that $$x=z_1~(\rho \cap \mathcal{F})~z_2~(\rho \cap \mathcal{F})~\cdots~(\rho \cap \mathcal{F})~z_{n-1}~(\rho \cap \mathcal{F})~z_n=e.$$ Let $g=(z_1^{-1}z_2)(z_2^{-1}z_3)\cdots(z_{n-1}^{-1}z_n)$. Then $g \in E_S$, $$xg=(z_1z_1^{-1})(z_2z_2^{-1})\cdots(z_{n-1}z_{n-1}^{-1})z_n \in E_S$$ and $x(xg)=x(xg)g=(xg)(xg)=xg=(xg)e=e(xg)$. From the foregoing and the fact that $xg\,\rho\,e$ we obtain that there exists $f=xg \in E_{e\rho}$ such that $xf=ef$, that is, $x \in \ker{\sigma_{e\rho}}$ and therefore also that $e\rho_t \subseteq \ker{\sigma_{e\rho}}$.
\end{proof}

We are now ready for characterisations of congruences having the property that $\rho_{kt}=\varepsilon$, $\rho_{tk}=\varepsilon$, or $\rho_{tkt}=\varepsilon$. The next four propositions are generalisations of the relating results in \cite{inverse} and \cite{wf2011}.

\begin{proposition}\label{k}
 The following statements concerning a congruence $\rho$ on an inverse semigroup $S$ are equivalent.
 \begin{enumerate}[label=(\arabic*)]
  \item $\rho_k=\varepsilon$;
  \item $\rho \subseteq \tau$;
  \item $\ker{\rho}=E_S$;
  \item $\rho$ is idempotent pure;
  \item $\ker{\rho}$ is a semilattice;
  \item for every $e \in E_S$, $e\rho$ is a semilattice;
  \item $\rho$ satisfies the implication $x^2\,\rho\,x \Rightarrow x^2=x$.
 \end{enumerate}
\end{proposition}
\begin{proof}
 The equivalence of (1), (2), (3), (4) and (7) follows from Lemma \ref{kreg}.

 $(3) \Rightarrow (5)$. It follows immediately from the fact that $E_S$ is a semilattice.

 $(5) \Rightarrow (6)$. For every $e \in E_S$, we have $e\rho \subseteq \ker{\rho}$. From the hypothesis and the fact that $e\rho$ is an inverse subsemigroup of $S$, we get that $e\rho$ is also a semilattice.

 $(6) \Rightarrow (3)$. If $x \in \ker{\rho}$, then $x\,\rho\,f$ for some $f \in E_S$, whence $x \in E_S$ by assumption. Hence $\ker{\rho} \subseteq E_S$ and $\ker{\rho}=E_S$.
\end{proof}

The next proposition is a generalisation of \cite[Exercises \@Roman3.6.9 (\@roman9)]{inverse}.

\begin{proposition}\label{kt}
 The following statements concerning a congruence $\rho$ on an inverse semigroup $S$ are equivalent.
 \begin{enumerate}[label=(\arabic*)]
  \item $\rho_{kt}=\varepsilon$;
  \item $\rho_k \subseteq \mu$;
  \item $\ker{\rho} \subseteq E_S\zeta$;
  \item $\ker{\rho}$ is a Clifford semigroup;
  \item for every $e \in E_S$, $e\rho$ is a Clifford semigroup;
  \item $\rho$ satisfies the implication $x^2\,\rho\,x$, $e \in E_S \Rightarrow xe=ex$.
 \end{enumerate}
\end{proposition}
\begin{proof}
 The equivalence of (1) and (2) follows directly from Proposition \ref{t}.

 $(2) \Rightarrow (3)$. Since $\rho_k \subseteq \mu$, we have $\ker{\rho_k} \subseteq \ker{\mu}$ and thus $\ker{\rho} \subseteq E_S\zeta$ by \cite[Lemma \@Roman3.3.5]{inverse}.

 $(3) \Rightarrow (4)$. Clearly, $E_S\zeta$ is a Clifford semigroup. But $\ker{\rho}$ is an inverse subsemigroup of $S$ and also an inverse subsemigroup of $E_S\zeta$. Hence $\ker{\rho}$ is a Clifford semigroup.

 $(4) \Rightarrow (5)$. For every $e \in E_S$, $e\rho$ is an inverse subsemigroup of $S$, and therefore, an inverse subsemigroup of $\ker{\rho}$. Now $\ker{\rho}$ is a Clifford semigroup and so is $e\rho$.

 $(5) \Rightarrow (6)$. If $x^2\,\rho\,x$, then the idempotent $\rho$-class $x\rho$ is a Clifford semigroup by assumption, and hence a union of groups. But $\ker{\rho}$ is a union of the idempotent $\rho$-classes. Then $\ker{\rho}$ is also a union of groups. This together with the fact that $\ker{\rho}$ is an inverse subsemigroup of $S$ gives that $\ker{\rho}$ is a Clifford semigroup. Therefore, $x$ commutes with idempotents.

 $(6) \Rightarrow (2)$. Suppose that $a \in \ker{\rho}$. Then $a^2\,\rho\,a$ and $a \in E_S\zeta$ by assumption. It follows that $\ker{\rho} \subseteq \ker{\mu}$ and thus also that $\rho_k \subseteq \mu$ by Lemma \ref{least}.
\end{proof}

The next proposition is a generalisation of \cite[Exercises \@Roman3.7.10 (\@roman2)]{inverse}.

\begin{proposition}\label{tk}
 The following statements concerning a congruence $\rho$ on an inverse semigroup $S$ are equivalent.
 \begin{enumerate}[label=(\arabic*)]
  \item $\rho_{tk}=\varepsilon$;
  \item for every $e \in E_S$, $e\rho$ is an $E$-unitary inverse semigroup;
  \item $\rho$ satisfies the implication $xy=y$, $x\,\rho\,y \Rightarrow x^2=x$.
 \end{enumerate}
\end{proposition}
\begin{proof}
 $(1) \Rightarrow (2)$. Notice that $e\rho_t$ is a semilattice by Proposition \ref{k}. Next, $e\rho_t=\ker{\sigma_{e\rho}}$ by Proposition \ref{kernel}. It follows that the kernel of the least group congruence on $e\rho$ is a semilattice and so, by \cite[Proposition \@Roman3.7.2]{inverse}, $e\rho$ is an $E$-unitary inverse semigroup, which gives us the result we require.

 $(2) \Rightarrow (3)$. Let $x$, $y \in S$ be such that $xy=y$ and $x\,\rho\,y$. Then $x\,\rho\,y=xy\,\rho\,xx$ while $x \in \ker{\rho}$ and $x\,\rho\,xx^{-1}\,\rho\,yy^{-1}$. Furthermore, $x(yy^{-1})=yy^{-1}=(yy^{-1})(yy^{-1})$ so that $x\,\sigma_{(yy^{-1})\rho}\,yy^{-1}$ whence $x \in \ker{\sigma_{(yy^{-1})\rho}}$. It follows that $x^2=x$ since $(yy^{-1})\rho$ is $E$-unitary.

 $(3) \Rightarrow (1)$. Suppose that $a \in \ker{\rho_t}$. Then $a\,\rho_t\,e$ for some $e \in E_S$. By Result \ref{rholt}, $af=ef$ and $f\,\rho\,a^{-1}a\,\rho\,e$ for some $f \in E_S$. So we have $a(ef)=(af)e=(ef)e=ef$, and this together with $ef\,\rho\,e\,\rho\,a$ and the assumption gives $a \in E_S$. Thus $\ker{\rho_t}=E_S$ and so $\rho_{tk}=\varepsilon$.
\end{proof}

The next proposition is a generalisation of \cite[Proposition 2.2]{wf2011}.

\begin{proposition}\label{tkt}
The following statements concerning a congruence $\rho$ on an inverse semigroup $S$ are equivalent.
\begin{enumerate}[label=(\arabic*)]
 \item $\rho_{tkt}=\varepsilon$;
 \item for every $e \in E_S$, $e\rho$ is an $E\omega$-Clifford semigroup;
 \item $\rho$ satisfies the implication $xy=y$, $x\,\rho\,y\,\rho\,e$, $e \in E_S \Rightarrow xe=ex$.
\end{enumerate}
\end{proposition}
\begin{proof}
 $(1) \Rightarrow (2)$. By Proposition \ref{kt} we can deduce that $e\rho_t$ is a Clifford semigroup. But $e\rho_t=\ker{\sigma_{e\rho}}$ by Proposition \ref{kernel} and $\ker{\sigma_{e\rho}}=E_{e\rho}\omega$ by \cite[Lemma \@Roman3.5.7]{inverse}, which implies that $E_{e\rho}\omega$ is a Clifford semigroup and therefore $e\rho$ is an $E\omega$-Clifford semigroup.

 $(2) \Rightarrow (3)$. Let $x$, $y \in S$ and $e \in E_S$ be such that $xy=y$ and $x\,\rho\,y\,\rho\,e$. Then $x(yy^{-1})=(yy^{-1})(yy^{-1})$ and $yy^{-1}\,\rho\,ee^{-1}=e$, which gives that $x \in \ker{\sigma_{e\rho}}=E_{e\rho}\omega$. But $E_{e\rho}\omega$ is a Clifford semigroup and $e$ is an idempotent in $E_{e\rho}\omega$. So we have $xe=ex$.

 $(3) \Rightarrow (1)$. Assume that $e \in E_S$, $x \in \ker{\sigma_{e\rho}}$ and $f \in E_{e\rho}$. Then $x\,\sigma_{e\rho}\,g$ for some $g \in E_{e\rho}$ whence $xh=gh$ for some $h \in E_{e\rho}$. So $x(gh)=x(hg)=(gh)g=gh$. Further, $gh\,\rho\,e\,\rho\,x\,\rho\,f$. By the hypothesis this implies that $xf=fx$ whence $\ker{\sigma_{e\rho}}$ is a Clifford semigroup. By Proposition \ref{kernel}, $e\rho_t=\ker{\sigma_{e\rho}}$ which implies that $e\rho_t$ is a Clifford semigroup. Therefore $\rho_{tkt}=\varepsilon$ now follows from Proposition \ref{kt}.
\end{proof}

Recall that an inverse semigroup $S$ is \emph{$E$-reflexive} if for any $x$, $y \in S$ and $e \in E_S$, $xey \in E_S$ implies $yex \in E_S$. Equivalently, $S$ is $E$-reflexive if and only if every $\eta$-class of $S$ is $E$-unitary. For congruences satisfying $\rho_{ktk}=\varepsilon$, we have the following result.

\begin{proposition}\label{ktk}
 Let $\rho$ be a congruence on an inverse semigroup such that $\rho_{ktk}=\varepsilon$. Then
 \begin{enumerate}[label=(\arabic*)]
  \item for every $e \in E_S$, $e\rho_k$ is an $E$-unitary inverse semigroup;
  \item $\ker{\rho}$ is $E$-reflexive.
 \end{enumerate}
\end{proposition}
\begin{proof}
 (1) This is an immediate consequence of Proposition \ref{tk}.

 (2) By Proposition \ref{tk}, $e\rho_k$ is $E$-unitary for every $e \in E_S$, and therefore $\ker{\rho}=\ker{\rho_k}$ is a semilattice of $E$-unitary inverse semigroups, which is induced by the following homomorphism, $$(\rho_k)^{\natural}|_{\ker{\rho}}:~\ker{\rho} \longrightarrow E_{S/\rho_k}.$$ It follows that $\ker{\rho}$ is $E$-reflexive.
\end{proof}

We are now ready for relationships among quotient semigroups by the min network of congruences.

\begin{theorem}\label{rel}
 Let $S$ be an inverse semigroup, and let $n$ be a positive integer. Then
 \begin{enumerate}[label=(\arabic*)]
  \item every idempotent $\alpha_n$-class of $S/\beta_{n+1}$ is a semilattice, and every idempotent $\beta_n$-class of $S/\alpha_{n+1}$ is a group;
  \item $\ker{(\alpha_n)_{S/\alpha_{n+2}}}$ is a Clifford semigroup, and every idempotent $\beta_n$-class of $S/\beta_{n+2}$ is an $E$-unitary inverse semigroup;
  \item $\ker{(\alpha_n)_{S/\beta_{n+3}}}$ is an $E$-reflexive inverse semigroup, and every idempotent $\beta_n$-class of $S/\alpha_{n+3}$ is an $E\omega$-Clifford semigroup.
 \end{enumerate}
\end{theorem}
\begin{proof}
 By Proposition \ref{ab} and Proposition \ref{min} we have
 \begin{center}
  $((\alpha_n)_{S/\beta_{n+1}})_k=(\alpha_n/\beta_{n+1})_k=(\alpha_n)_k/\beta_{n+1}=\beta_{n+1}/\beta_{n+1}=\varepsilon_{S/\beta_{n+1}}$,\\
  $((\beta_n)_{S/\alpha_{n+1}})_t=(\beta_n/\alpha_{n+1})_t=(\beta_n)_t/\alpha_{n+1}=\alpha_{n+1}/\alpha_{n+1}=\varepsilon_{S/\alpha_{n+1}}$;\\
  $((\alpha_n)_{S/\alpha_{n+2}})_{kt}=(\alpha_n/\alpha_{n+2})_{kt}=(\alpha_n)_{kt}/\alpha_{n+2}=\alpha_{n+2}/\alpha_{n+2}=\varepsilon_{S/\alpha_{n+2}}$,\\
  $((\beta_n)_{S/\beta_{n+2}})_{tk}=(\beta_n/\beta_{n+2})_{tk}=(\beta_n)_{tk}/\beta_{n+2}=\beta_{n+2}/\beta_{n+2}=\varepsilon_{S/\beta_{n+2}}$;\\
  $((\alpha_n)_{S/\beta_{n+3}})_{ktk}=(\alpha_n/\beta_{n+3})_{ktk}=(\alpha_n)_{ktk}/\beta_{n+3}=\beta_{n+3}/\beta_{n+3}=\varepsilon_{S/\beta_{n+3}}$,\\
  $((\beta_n)_{S/\alpha_{n+3}})_{tkt}=(\beta_n/\alpha_{n+3})_{tkt}=(\beta_n)_{tkt}/\alpha_{n+3}=\alpha_{n+3}/\alpha_{n+3}=\varepsilon_{S/\alpha_{n+3}}$.
 \end{center}
 The desired results then follow immediately from Propositions \ref{k}, \ref{t}, \ref{kt}, \ref{tk}, \ref{ktk}, and \ref{tkt}.
\end{proof}

\begin{remark}
 It follows directly from (2) of Theorem \ref{rel} that $\alpha_{n+2}$ is a $\ker{\alpha_n}$-is-Clifford congruence, and that $\beta_{n+2}$ is a $\beta_n$-is-over-$E$-unitary congruence, which are due to Theorem 2.9 and Theorem 2.13 of \cite{fw2019}.
\end{remark}

We next proceed to explore further relationships among quotient semigroups by the min network of congruences. Again we need some preparation.

\begin{lemma}\label{generate}
 Let $\mathsf{R}$ be a relation on an arbitrary semigroup $X$, $Y$ be a subsemigroup of $X$, and $(\mathsf{R}|_Y)^*$ be the congruence on $Y$ generated by $\mathsf{R}|_Y$. Then $(\mathsf{R}|_Y)^* \subseteq \mathsf{R}^*|_Y$.
\end{lemma}
\begin{proof}
 Let $\mathsf{R}^s$, $\mathsf{R}^c$, and $\mathsf{R}^\infty$ be the smallest reflexive and symmetric, smallest left and right compatible, and smallest transitive relation on $X$ containing $\mathsf{R}$ respectively. And let $(\mathsf{R}|_Y)^s$, $(\mathsf{R}|_Y)^c$, and $(\mathsf{R}|_Y)^\infty$ be the smallest reflexive and symmetric, smallest left and right compatible, and smallest transitive relation on $Y$ containing $\mathsf{R}|_Y$ respectively. Then we have $$(\mathsf{R}|_Y)^s \subseteq \mathsf{R}^s|_Y, \quad (\mathsf{R}|_Y)^c \subseteq \mathsf{R}^c|_Y, \quad (\mathsf{R}|_Y)^\infty \subseteq \mathsf{R}^\infty|_Y,$$ since
 \begin{eqnarray*}
  (\mathsf{R}|_Y)^s&=&\mathsf{R}|_Y \cup (\mathsf{R}|_Y)^{-1} \cup \varepsilon_Y=(\mathsf{R} \cup \mathsf{R}^{-1} \cup \varepsilon_X)|_Y=\mathsf{R}^s|_Y,\\
  (\mathsf{R}|_Y)^c&=&\{(xay,xby) \in Y \times Y\,|\,x,y \in Y^1,~(a,b) \in \mathsf{R}|_Y\}\\
  &\subseteq& \{(xay,xby) \in Y \times Y\,|\,x,y \in X^1,~ (a,b) \in \mathsf{R}\}\\
  &=&\{(xay,xby) \in X \times X\,|\,x,y \in X^1, ~(a,b) \in \mathsf{R}\}|_Y\\
  &=&\mathsf{R}^c|_Y,\\
  (\mathsf{R}|_Y)^\infty&=&\bigcup_{n \in \mathbb{Z}^+}(\mathsf{R}|_Y)^n \subseteq \bigcup_{n \in \mathbb{Z}^+}\left(\mathsf{R}^n|_Y\right)=\left(\bigcup_{n \in \mathbb{Z}^+} \mathsf{R}^n\right)\Bigg|_Y=\mathsf{R}^\infty|_Y.
 \end{eqnarray*}
 Consequently,
 $$(\mathsf{R}|_Y)^*=(\mathsf{R}|_Y)^{sc\infty} \subseteq (\mathsf{R}^s|_Y)^{c\infty} \subseteq (\mathsf{R}^{sc}|_Y)^\infty \subseteq \mathsf{R}^{sc\infty}|_Y=\mathsf{R}^*|_Y.$$
\end{proof}

\begin{lemma}\label{restriction}
 Let $\rho$ be a congruence on an inverse semigroup $S$ and $e \in E_S$. Then $$\mathcal{L}_S|_{e\rho}=\mathcal{L}_{e\rho}, \qquad \mathcal{F}_S|_{e\rho}=\mathcal{F}_{e\rho}, \qquad \rho|_{e\rho}=\omega_{e\rho}.$$
\end{lemma}
\begin{proof}
 Assume that $a$, $b \in e\rho$. If $a\,\mathcal{L}_S\,b$, then $a^{-1}$, $b^{-1}$ and $a^{-1}a=b^{-1}b$ are all $\rho$-related to $e$, which implies that $a\,\mathcal{L}_{e\rho}\,b$ whence $\mathcal{L}_S|_{e\rho} \subseteq \mathcal{L}_{e\rho}$. The reverse inclusion $\mathcal{L}_{e\rho} \subseteq \mathcal{L}_S|_{e\rho}$ is obvious. Therefore $\mathcal{L}_S|_{e\rho}=\mathcal{L}_{e\rho}$.

 If $a\,\mathcal{F}_S\,b$, then $a^{-1}b \in E_S$. Now $a^{-1} \in e\rho$. Thus $a^{-1}b \in e\rho$ and $a^{-1}b \in E_{e\rho}$ whence $a\,\mathcal{F}_{e\rho}\,b$, which implies that $\mathcal{F}_S|_{e\rho} \subseteq \mathcal{F}_{e\rho}$. Clearly, $\mathcal{F}_{e\rho} \subseteq \mathcal{F}_S|_{e\rho}$. Hence $\mathcal{F}_S|_{e\rho}=\mathcal{F}_{e\rho}$.

 Finally $\rho|_{e\rho}=\omega_{e\rho}$ since any two elements in $e\rho$ are $\rho$-related.
\end{proof}

\begin{lemma}\label{omega}
 Let $\rho$ be a congruence on an inverse semigroup $S$ and let $n$ be a positive integer. Then for any $x_1$, $x_2$, $\cdots$, $x_n \in \{t,k\}$ and $e \in E_S$, $$(\omega_{e\rho})_{x_1 x_2\,\cdots\,x_n} \subseteq \rho_{x_1 x_2\,\cdots\,x_n}|_{e\rho}.$$
\end{lemma}
\begin{proof}
 Let $$\mathcal{A}_i=\begin{cases}\mathcal{F}, &x_i=t,\\ \mathcal{L}, &x_i=k,\end{cases}$$ where $i=1$, $\cdots$, $n$. It follows from Result \ref{tkgen}, Lemma \ref{generate} and Lemma \ref{restriction} that
 \begin{eqnarray*}
  (\omega_{e\rho})_{x_1x_2\,\cdots\,x_n}&=&(\cdots((\rho|_{e\rho} \bigcap (\mathcal{A}_1)_{e\rho})^*\bigcap (\mathcal{A}_2)_{e\rho})^* \bigcap \cdots \bigcap (\mathcal{A}_n)_{e\rho})^*\\
  &=&(\cdots(((\rho \bigcap \mathcal{A}_1)|_{e\rho})^* \bigcap \mathcal{A}_2|_{e\rho})^* \bigcap \cdots \bigcap \mathcal{A}_n|_{e\rho})^*\\
  &\subseteq& (\cdots ((\rho \bigcap \mathcal{A}_1)^*|_{e\rho} \bigcap \mathcal{A}_2|_{e\rho})^* \bigcap \cdots \bigcap \mathcal{A}_n|_{e\rho})^*\\
  &=&(\cdots (((\rho \bigcap \mathcal{A}_1)^* \bigcap \mathcal{A}_2)|_{e\rho})^* \bigcap \cdots \bigcap \mathcal{A}_n|_{e\rho})^*\\
  &\vdots&\\
  &\subseteq& ((\cdots ((\rho \bigcap \mathcal{A}_1)^* \bigcap \mathcal{A}_2)^* \bigcap \cdots \bigcap \mathcal{A}_n)|_{e\rho})^*\\
  &\subseteq& (\cdots((\rho \bigcap \mathcal{A}_1)^* \bigcap \mathcal{A}_2)^* \bigcap \cdots \bigcap \mathcal{A}_n)^*|_{e\rho}\\
  &=&\rho_{x_1 x_2 \cdots\,x_n}|_{e\rho},
 \end{eqnarray*}
 as required.
\end{proof}

\begin{remark}
 Lemma \ref{omega} has a more general form: for any $x_1$, $\cdots$, $x_m$, $\cdots$, $x_n \in \{t,k\}$ and $e \in E_S$, $$(\omega_{e\rho_{x_1x_2\cdots x_m}})_{x_{m+1}\cdots x_n} \subseteq \rho_{x_1 \cdots x_n}|_{e\rho_{x_1 \cdots x_m}}.$$
\end{remark}

The second main result of this section may now be established.

\begin{theorem}\label{main}
 Let $n$ be a non-negative integer. The following statements are valid in any inverse semigroup $S$.
 \begin{enumerate}[label=(\arabic*)]
  \item Every $\eta$-class of $S/\beta_{2n+3}$ is a $\beta_{2n}$-is-over-$E$-unitary semigroup;
  \item every $\eta$-class of $S/\alpha_{2(n+2)}$ is a $\ker{\alpha_{2n+1}}$-is-Clifford semigroup;
  \item $(E_{S/\beta_{2(n+2)}})\omega$ is a $\beta_{2n+1}$-is-over-$E$-unitary semigroup;
  \item $(E_{S/\alpha_{2n+3}})\omega$ is a $\ker{\alpha_{2n}}$-is-Clifford semigroup.
 \end{enumerate}
\end{theorem}
\begin{proof}
 (1) Denote $S/\beta_{2n+3}$ by $Z$. Let $e \in E_Z$. From Lemma \ref{omega} it follows that $$(\beta_{2(n+1)})_{e\eta_Z}=(\omega_{e\eta_Z})_{(tk)^{n+1}} \subseteq (\eta_Z)_{(tk)^{n+1}}|_{e\eta_Z}=(\beta_{2n+3})_Z|_{e\eta_Z}.$$ But Proposition \ref{ab} gives $$(\beta_{2n+3})_Z=\beta_{2n+3}/\beta_{2n+3}=\varepsilon_{S/\beta_{2n+3}}=\varepsilon_Z.$$ As a result we see that $$(\beta_{2(n+1)})_{e\eta_Z} \subseteq \varepsilon_Z|_{e\eta_Z},$$ and hence $(\beta_{2(n+1)})_{e\eta_Z}=\varepsilon_Z|_{e\eta_Z}$. By Proposition \ref{ktk}, $(\beta_{2n})_{e\eta_Z}$ is over $E$-unitary inverse semigroups.

 (2) Denote $S/\alpha_{2(n+2)}$ by $Z$. By Proposition \ref{kt}, to show that $\ker{(\alpha_{2n+1})_{e\eta_Z}}$ is a Clifford semigroup, it is sufficient to prove that $((\alpha_{2n+1})_{e\eta_Z})_{kt}=\varepsilon_{e\eta_Z}$, that is, $(\alpha_{2n+3})_{e\eta_Z}=\varepsilon_{e\eta_Z}$.

 From Lemma \ref{omega} we find that $$(\alpha_{2n+3})_{e\eta_Z}=(\omega_{e\eta_Z})_{(tk)^{n+1}t} \subseteq (\eta_Z)_{(tk)^{n+1}t}|_{e\eta_Z}=(\alpha_{2(n+2)})_Z|_{e\eta_Z}.$$ By Proposition \ref{ab}, $$(\alpha_{2(n+2)})_Z=\alpha_{2(n+2)}/\alpha_{2(n+2)}=\varepsilon_{S/\alpha_{2(n+2)}}=\varepsilon_Z,$$ which implies that $(\alpha_{2n+3})_{e\eta_Z} \subseteq \varepsilon_Z|_{e\eta_Z}=\varepsilon_{e\eta_Z}$ and therefore the equality $(\alpha_{2n+3})_{e\eta_Z}=\varepsilon_{e\eta_Z}$ takes place.

 (3) Denote $S/\beta_{2(n+2)}$ by $Z$. Notice that $E_Z\omega=\ker{\sigma_Z}=e\sigma_Z$ for every $e \in E_Z$. Let $f \in E_Z$. In view of Proposition \ref{ktk}, to show that $f(\beta_{2n+1})_{E_Z\omega}$ is an $E$-unitary inverse semigroup, it suffices to prove that $((\alpha_{2n})_{e\sigma_Z})_{ktk}=\varepsilon_{e\sigma_Z}$, that is, $(\beta_{2n+3})_{e\sigma_Z}=\varepsilon_{e\sigma_Z}$.

 We infer from Lemma \ref{omega} that $$(\beta_{2n+3})_{e\sigma_Z}=(\omega_{e\sigma_Z})_{(kt)^{2(n+1)}k} \subseteq (\sigma_Z)_{(kt)^{2(n+1)}k}|_{e\sigma_Z}=(\beta_{2(n+2)})_Z|_{e\sigma_Z}.$$ Furthermore, Proposition \ref{ab} gives $$(\beta_{2(n+2)})_Z=\beta_{2(n+2)}/\beta_{2(n+2)}=\varepsilon_{S/\beta_{2(n+2)}}=\varepsilon_Z,$$ and therefore also $(\beta_{2n+3})_{e\sigma_Z} \subseteq \varepsilon_Z|_{e\sigma_Z}=\varepsilon_{e\sigma_Z}$.

 (4) Denote $S/\alpha_{2n+3}$ by $Z$. Note that for every $e \in E_Z$, $E_Z\omega=\ker{\sigma_Z}=e\sigma_Z$. By Proposition \ref{kt}, in order to prove that $\ker{(\alpha_{2n})_{E_Z\omega}}=\ker{(\alpha_{2n})_{e\sigma_Z}}$ is a Clifford semigroup, it is sufficient to show that $((\alpha_{2n})_{e\sigma_Z})_{kt}=\varepsilon_{e\sigma_Z}$, that is, $(\alpha_{2(n+1)})_{e\sigma_Z}=\varepsilon_{e\sigma_Z}$.

 From Lemma \ref{omega} we derive that $$(\alpha_{2(n+1)})_{e\sigma_Z}=(\omega_{e\sigma_Z})_{(kt)^{n+1}} \subseteq (\sigma_Z)_{(kt)^{n+1}}|_{e\sigma_Z}=(\alpha_{2n+3})_Z|_{e\sigma_Z}.$$ However, Proposition \ref{ab} implies that $$(\alpha_{2n+3})_Z=\alpha_{2n+3}/\alpha_{2n+3}=\varepsilon_{S/\alpha_{2n+3}}=\varepsilon_Z,$$ whence $(\alpha_{2(n+1)})_{e\sigma_Z} \subseteq \varepsilon_Z|_{e\sigma_Z}$ and so the proof is complete.
\end{proof}

\begin{remark}
 Theorem \ref{main} determines exactly the relationships among the quasivarieties in which the quotient semigroups lie. Take $\ker{\alpha_{2(n+1)}}$-is-Clifford semigroups for example. If $S$ is a $\ker{\alpha_{2(n+1)}}$-is-Clifford semigroup, then $\alpha_{2(n+2)}$, the least $\ker{\alpha_{2(n+1)}}$-is-Clifford congruence on $S$, is the identity relation. By (2) of Theorem \ref{main}, every $\eta$-class of $S$ ($=S/\varepsilon$) is a $\ker{\alpha_{2n+1}}$-is-Clifford semigroup. That is to say, every $\eta$-class of a $\ker{\alpha_{2(n+1)}}$-is-Clifford semigroup is a $\ker{\alpha_{2n+1}}$-is-Clifford semigroup. So Theorem \ref{main} is exactly the response to Problem 5 in \cite{wf2011}. Every $\eta$-class of a Clifford semigroup is a group, every $\eta$-class of an $E$-reflexive inverse semigroup is $E$-unitary, every $\eta$-class of a $\ker{\nu}$-is-Clifford semigroup is a $\ker{\sigma}$-is-Clifford semigroup, $\cdots$, and the closure of the set of idempotents, or $E\omega$, of an $E$-unitary inverse semigroup is a semilattice, $E\omega$ of a $\ker{\sigma}$-is-Clifford semigroup is a Clifford semigroup, $E\omega$ of a $\pi$-is-over-$E$-unitary semigroup is $E$-reflexive, $\cdots$. These patterns continue indefinitely.
\end{remark}

\noindent\textbf{Acknowledgements} The authors are grateful to the referee for his/her very careful reading of their manuscript, and the very detailed and helpful comments provided. They would also like to thank Professor Victoria Gould for the help with English and expressions, which has resulted in a much improved article. This work is supported by the National Natural Science Foundation of China (Grant No.: 11901088, 11871150) and China Scholarship Council.

\end{document}